\documentclass[amssymb,11pt]{amsart}
\usepackage{amscd,amssymb,euscript,amsthm}
\theoremstyle{plain}
\newtheorem{thm}{Theorem}[section] 
\newtheorem{cor}[thm]{Corollary}

\newtheorem{lem}[thm]{Lemma}

\theoremstyle{definition}
\newtheorem{defn}[thm]{Definition}
\theoremstyle{remark}
\newtheorem{rem}[thm]{Remark}

\numberwithin{equation}{section}
\def\<{\left<}
\def\>{\right>}
\def\cstar{$C^*$-algebra}
\begin{document}
\title[Operator systems in matrix algebras]{The noncommutative Choquet boundary III:
\\Operator systems 
in matrix algebras}
\author{William Arveson}
%
%
\address{Department of Mathematics,
University of California, Berkeley, CA 94720}
\email{arveson@math.berkeley.edu}
%
\date{25 October, 2008}

\begin{abstract}  We classify operator systems 
$S\subseteq \mathcal B(H)$ 
that act on finite dimensional Hilbert spaces $H$  
by making use of the noncommutative Choquet boundary.  
$S$ is said to be {\em reduced} when its boundary ideal is $\{0\}$.  
In the category of operator systems, that property 
functions as semisimplicity does in the category of 
complex Banach algebras.    

We construct explicit examples of reduced 
operator systems using sequences of ``parameterizing maps" 
$\Gamma_k: \mathbb C^r\to \mathcal B(H_k)$, $k=1,\dots, N$.  We show 
that every reduced operator system is isomorphic to one 
of these, and that two 
sequences give rise to isomorphic operator systems if and only if 
they are ``unitarily equivalent" 
parameterizing sequences.  

Finally, we construct nonreduced operator systems $S$ that have a 
given boundary ideal $K$ and a given 
reduced image in $C^*(S)/K$, and  
show that these constructed examples 
exhaust the possibilities.  
\end{abstract}

\maketitle

\section{Introduction}\label{S:in}

This paper continues the series \cite{arvChoq} and \cite{arvChoq2} by addressing the 
problem of classifying operator spaces that generate finite dimensional \cstar s.  While 
this is a restricted class of operator spaces in which many 
subtle topological 
obstructions disappear, one can also argue that it contains all  
of the noncommutativity.  
Correspondingly, our classification results will make 
essential use of the noncommutative Choquet boundary.  

An {\em operator space} is a norm closed linear subspace of the algebra $\mathcal B(H)$ of bounded 
operators on a Hilbert space $H$.  Operator spaces are the objects of a category that refines the 
classical category of Banach spaces in a significant way, 
the refinement being that morphisms in the category 
of operator spaces are completely bounded linear maps rather than
bounded linear maps.  In the operator space category, 
maps are typically endowed with their completely bounded norm.  
If one restricts to the smaller category of operator spaces with {\em completely contractive}  
linear maps, one obtains a noncommutative refinement of the category of Banach spaces in 
which the term {\em classification} means classifying operator 
spaces up to completely isometric isomorphism.  As we have said above, 
this paper addresses (somewhat indirectly) the 
problem of classifying operator spaces that can be realized as subspaces of $\mathcal B(H)$ 
where $H$ is a {\em finite dimensional} Hilbert space.  

We say indirectly because we do not deal directly with operator spaces below, but rather with operator 
{\em systems}.  An operator system is an operator space that is closed under the $*$-operation of 
$\mathcal B(H)$ and which contains the identity operator.  Throughout the matrix hierarchy over 
an operator system $S$ there are enough positive operators to generate the space of self adjoint matrices 
over $S$.  Correspondingly, the morphisms of the category of operator 
systems are the unit-preserving completely positive (UCP) linear maps.  
An {\em isomorphism} of operator systems $S_1$, $S_2$ is a UCP map $\phi:S_1\to S_2$ 
that has a UCP inverse $\phi^{-1}:S_2\to S_1$; and it is known that this is equivalent to the existence 
of a completely isometric linear map of $S_1$ on $S_2$ that carries the unit of $S_1$ to 
the unit of $S_2$. 
 
Paulsen's device (see  p. 104 of \cite{paulsenBk2} or p. 21 of \cite{BlLeMbook}) allows one 
to associate an operator system $\tilde S\subseteq \mathcal B(H\oplus H)$ 
with an arbitrary operator space $S\subseteq \mathcal B(H)$ in the following way 
$$
\tilde S=\{
\begin{pmatrix}
\lambda\cdot\mathbf 1& s\\t^*&\mu\cdot\mathbf 1
\end{pmatrix}
: s,t\in S,\ \lambda, \mu\in\mathbb C\}, 
$$
and this association of $\tilde S$ with $S$ is functorial in that 
completely contractive maps of operator spaces $S$ give rise to 
UCP maps of operator systems $\tilde S$.  In this way, many if not 
most results about operator systems lead directly 
to results in the somewhat broader category of operator spaces.  Consequently, 
we shall work exclusively with operator 
systems throughout this paper.

The fundamental fact about general operator systems 
is that there is a {\em largest} closed two-sided ideal $K$ in the 
$C^*$-subalgebra $C^*(S)$ generated by $S$ such that the natural map 
$x\in C^*(S)\mapsto \dot x\in C^*(S)/K$ is completely isometric on $S$.  
In this paper we will 
follow \cite{arvChoq2} by referring 
to this ideal $K$ as the {\em boundary ideal} for $S$.  The associated embedding 
$\dot S\subseteq C^*(S)/K$ is called the $C^*$-envelope of $S$.  

\begin{defn}\label{inDef1}
An operator system $S\subseteq C^*(S)$ 
is said to be {\em reduced} if its boundary ideal is $\{0\}$. 
\end{defn}
We have found it useful to view the property of Definition \ref{inDef1} as the proper counterpart for 
operator systems of the semisimplicity property of complex algebras, especially 
for operator systems in matrix algebras.  Thus, the boundary ideal functions 
for operator systems 
in much the same way as the radical does for complex algebras.  

Our main results address the problem of classifying 
reduced operator systems $S\subseteq \mathcal B(H)$ 
that act on a finite dimensional Hilbert space $H$, and can be summarized 
as follows.  The \cstar\ generated by such an operator system decomposes uniquely 
into a central direct sum of matrix algebras 
\begin{equation}\label{inEq0}
C^*(S)=A_1\oplus\cdots\oplus A_N, \qquad A_k\cong \mathcal B(H_k)
\end{equation}
and it has several integer invariants associated with it: 
the dimension $d$ of $S$ itself, the number 
$N$ of mutually inequivalent irreducible representations 
of its generated \cstar\  $\pi_k: C^*(S)\to \mathcal B(H_k)$, 
and the dimensions $n_k=\dim H_k$, $1\leq k\leq N$ of these representations.  These numbers satisfy only 
one obvious constraint, namely
\begin{equation}\label{inEq1}
\dim \pi_k(S)\leq n_k^2, \qquad k=1,\dots, N,   
\end{equation}
and consequently $d\leq n_1^2+\cdots+n_N^2$.  

Given a set of numbers $d, N, n_1,\dots, n_N$, we first show how one 
constructs examples of reduced operator systems $S$ 
of dimension $d$ that have these integer invariants.  
By a $*$-vector space we mean a finite dimensional 
complex vector space that has been endowed with a 
distinguished conjugation (an antilinear map $z\mapsto \bar z$ satisfying $\bar{\bar z}=z$), 
and a distinguished ``unit" $\mathbf 1$ - a nonzero self adjoint element of $Z$.  
Nothing is lost if one thinks of $Z$ as the $*$-vector space $\mathbb C^d$ with involution 
$(z_1, \dots, z_d)^*=(\bar z_1, \dots, \bar z_d)$ and 
unit $\mathbf 1=(1,1, \dots, 1)$.   For every $k=1,\dots, N$ let $H_k$ be a 
Hilbert space of dimension $n_k$, and let 
$$
\Gamma_k: Z\to \mathcal B(H_k) 
$$
be a linear map that preserves the $*$-operation and maps the unit of $Z$ to the identity 
operator of $\mathcal B(H_k)$.  We assume 
further that these maps $\Gamma_k$ satisfy the following three requirements: 
\begin{enumerate}
\item[(i)] {\em Irreducibility}: For each $k=1,\dots, N$,  
$\Gamma_k(Z)$ is an irreducible space of operators in $\mathcal B(H_k)$.  
\item[(ii)] {\em Faithfulness:} $\ker\Gamma_1\cap\cdots\cap\ker\Gamma_N=\{0\}$.  
\item[(iii)] {\em Strong separation}:  For every $k=1,\dots, N$, there is a positive integer 
$p=p(k)$ and a $p\times p$ matrix $(z_{ij})=(z_{ij}(k))\in M_p(Z)$ with entries in $Z$ such that 
\begin{equation}\label{inEq2}
\|(\Gamma_k(z_{ij}))\|>\max_{j\neq k}\|(\Gamma_j(z_{ij}))\|, \qquad k=1,2,\dots, N.  
\end{equation}
\end{enumerate}
Property (i) simply means that $C^*(\Gamma_k(Z))=\mathcal B(H_k)$, and property (ii) 
can be arranged in general 
by replacing $A=\mathbb C^d$ with a suitably smaller parameter space, 
if necessary.    
Property (iii) is critical, and as we shall see, it connects with the 
noncommutative Choquet boundary in an essential way.  

The $N$-tuple of operator mappings ${\bf \Gamma}=(\Gamma_1,\dots, \Gamma_N)$ gives rise 
to an operator system $S_{\bf \Gamma}$ acting on $H_1\oplus\cdots\oplus H_N$ as follows
\begin{equation}\label{inEq3}
S_{\bf \Gamma}=\{\Gamma_1(z)\oplus\cdots\oplus\Gamma_N(z): z\in Z\},   
\end{equation}
and by property (ii), we have $\dim(S_{\bf \Gamma})=\dim Z=d$.  

\begin{thm}\label{inThm1}
The operator system $S_{\bf \Gamma}$ is reduced, the center 
of $C^*(S_{\bf \Gamma})$ is $\mathbb C^N$,   
and the numbers $n_1, \dots, n_N$ are given by $n_k=\dim H_k$ for $k=1,\dots, N$.  
Conversely, every reduced operator system with the same integer invariants is isomorphic to 
an $S_{\bf \Gamma}$ constructed in this way from an $N$-tuple  of linear maps 
$\Gamma_k: Z\to \mathcal B(\tilde H_k)$, $k=1,\dots, N$ satisfying properties 
(i), (ii), (iii).    
\end{thm}

Theorem \ref{inThm1} reduces the classification problem for reduced operator systems 
acting on finite dimensional Hilbert spaces to the problem of 
determining when two sequences of parameterizing maps give rise to isomorphic 
operator systems.  Notice that the operator space $S_{\bf \Gamma}$ 
does not change if we compose 
each map of 
the sequence ${\bf \Gamma}=(\Gamma_1,\dots,\Gamma_N)$ with a single automorphism of the unital $*$ structure 
of the parameter space $Z$.  Moreover, it is easy to check 
that the {\em isomorphism class} of $S_{\bf \Gamma}$ 
does not change if we replace each $\Gamma_k$ with a unitarily equivalent 
map $\tilde\Gamma_k$, or if we permute the component maps $\Gamma_1,\dots, \Gamma_N$.  
Thus we say that two parameterizing 
sequences are {\em equivalent} if they have the same length 
${\bf \Gamma}=(\Gamma_1,\dots, \Gamma_N)$ and 
$\tilde{\bf \Gamma}=(\tilde\Gamma_1,\dots, \tilde\Gamma_N)$, 
there is a unit-preserving $*$-automorphism $\theta: Z\to Z$ of the parameter 
space, a permutation $\sigma$ of $\{1,2,\dots, N\}$,  and a sequence 
of unitary operators $U_k: H_k\to \tilde H_{\sigma(k)}$, $k=1,\dots, N$,  such that 
$$
\tilde\Gamma_{\sigma(k)}(\theta(z))=U_k\Gamma_k(z)U_k^{-1}, \qquad z\in Z, \quad k=1,\dots, N.  
$$
The following result completes the classification picture: 

\begin{thm}\label{inThm2} Let $S$ and $\tilde S$ be two reduced operator systems of the same 
dimension $d$ and having the same set of integer invariants.  Let $Z=\mathbb C^d$ 
and let ${\bf \Gamma}$ and ${\bf \tilde\Gamma}$ be two $N$-tuples of 
maps satisfying (i)--(iii) 
such that $S\cong S_{\bf \Gamma}$ and $\tilde S\cong S_{\bf \tilde\Gamma}$.  
Then $S$ and $\tilde S$ are isomorphic operator systems 
iff the two parameterizing sequences $\bf \Gamma$ and 
$\bf\tilde \Gamma$ are equivalent.  
\end{thm}

\begin{rem}[About the term ``classification"]  The 
combined statements of Theorems \ref{inThm1} and \ref{inThm2} amount to 
a classification of reduced operator systems that generate finite dimensional 
\cstar s.  It seems appropriate to offer some support for that claim,   
since the invariants of this classification, namely equivalence classes 
of parameterizing sequences ${\bf\Gamma}=(\Gamma_1,\dots, \Gamma_N)$ that satisfy 
properties (i), (ii) and (iii), are somewhat unusual.  

Consider first the 
simplest case $N=1$.  Here we have a single self adjoint 
linear map $\Gamma$ from the parameter 
space $Z=\mathbb C^d$ to operators on a finite dimensional Hilbert space $H$ 
such that $\Gamma(Z)$ is an irreducible set of operators, and which carries 
the unit of $Z$ to the identity operator.  Such a map $\Gamma$ automatically 
satisfies property (i), 
property (ii) can obviously be arranged, and property (iii) is 
satisfied vacuously.  Given a second map $\tilde\Gamma: Z\to \mathcal B(\tilde H)$ 
with that property, 
then $\Gamma$ and $\tilde\Gamma$ are equivalent iff there is a unitary 
operator $U: H\to \tilde H$ and an automorphism $\theta$ of the unital $*$-structure 
of $Z$ such that 
$$
\tilde\Gamma(\theta(z))=U\Gamma(z)U^{-1},\qquad z\in Z.  
$$
Thus, in the case $N=1$, Theorems \ref{inThm1} and \ref{inThm2} make 
the assertion that {\em two irreducible operator systems are isomorphic as operator 
systems iff they are unitarily equivalent}.  One can appreciate the content of 
that statement by considering that it has the following implication for two dimensional 
operator spaces: Given two 
irreducible $n\times n$ matrices $a$ and $b$, the map 
$$
\lambda\cdot\mathbf 1+\mu\cdot a\mapsto \lambda\cdot \mathbf 1+\mu\cdot b, 
\qquad \lambda,\mu\in \mathbb C
$$
is completely isometric iff the operators $a$ and $b$ are unitarily equivalent.  

In more 
complex situations where $N\geq 2$, the assertion is that a) every 
operator system is made up of irreducible ones with $N=1$, 
and b) there is no obstruction to assembling the irreducible pieces into a larger operator 
system other than the requirements of 
the strong separation property (iii) above.  In particular, the notion 
of {\em equivalence} for parameter sequences ${\bf\Gamma}=(\Gamma_1,\dots, \Gamma_N)$ 
primarily involves conditions on the individual 
coordinate maps $\Gamma_k$, with no 
interaction between different coordinates.  
\end{rem}

\section{Brief on the noncommutative Choquet boundary}\label{S:pc}

A {\em boundary representation} 
for an operator system $S\subseteq C^*(S)$ is an irreducible representation 
$\pi: C^*(S)\to \mathcal B(H)$ with the property that the only UCP map 
$\phi: C^*(S)\to \mathcal B(H)$ that agrees with $\pi$ on $S$ is $\pi$ itself.  
It was shown in \cite{arvChoq} that every separable operator system has 
sufficiently many boundary representations; and when compact operators are involved, 
it was shown in \cite{arvChoq2} that boundary representations are the 
noncommutative counterparts of peak points of function systems.  
In this section  we summarize that material 
in a form suitable for the analysis of operator systems 
in matrix algebras, referring the reader to \cite{arvChoq} and 
\cite{arvChoq2} for more 
detail and history.  While it is true that these results 
can be simplified very substantially for operator systems in 
matrix algebras, much of the subtlety persists even in the finite dimensional context
(see Remark \ref{pcRem1} below).  
Consequently, we have not attempted to make the following discussion of matrix 
algebras and their operator systems self-contained.  We now summarize 
some general results of \cite{arvSubalgI}, \cite{arvChoq} and \cite{arvChoq2} in the form 
we require.

Much of the discussion to follow rests on the main 
result of \cite{arvChoq} (Theorem 7.1), which we repeat here for reference: 

\begin{thm}\label{pcThm1} Every separable operator system $S\subseteq C^*(S)$  
has sufficiently many boundary representations 
in the sense that for every $n\geq 1$ and every $n\times n$ matrix 
$(s_{ij})$ with components $s_{ij}\in S$, one has 
\begin{equation}\label{pcEq1}
\|(s_{ij})\|=\sup_{\pi}\|(\pi(s_{ij}))\|, 
\end{equation}
the supremum on the right taken over all boundary representations $\pi$ for $S$. 
\end{thm} 

In Theorem 2.2.3 of \cite{arvSubalgI}, it was shown that in all cases where sufficiently many 
boundary representations exist, the boundary ideal is the intersection of 
the kernels of all boundary representations.  This leads to the following characterization 
of reduced operator systems in matrix algebras: 

\begin{cor}\label{pcCor1} An operator system $S\subseteq \mathcal B(H)$ acting on a finite 
dimensional Hilbert space $H$ is reduced iff every irreducible representation 
of $C^*(S)$ is a boundary representation for $S$.  
\end{cor}

\begin{proof}
If every irreducible representation is a boundary representation for $S$, then the result 
from \cite{arvSubalgI} cited above implies that the boundary ideal is trivial.  

Conversely, assume that the boundary ideal for $S$ is trivial and 
consider the decomposition (\ref{inEq0})  
\begin{equation}\label{pcEq2}
C^*(S)=A_1\oplus\cdots\oplus A_N, \qquad A_k\cong \mathcal B(H_k)
\end{equation}
of the \cstar\ generated by $S$ into a direct sum of full matrix algebras.  If one 
of the associated irreducible representations, say $\pi_k: C^*(S)\to \mathcal B(H_k)$, 
$1\leq k\leq N$,   
were {\em not} a boundary representation, then every boundary representation 
would annihilate the ideal 
$$
K=\bigcap_{j\neq k}\ker\pi_j\neq \{0\}
$$
and therefore would correspond to 
an irreducible representation of the quotient \cstar\ $C^*(S)/K$.  In that event, Theorem 
\ref{pcThm1} above implies that the 
quotient map $x\in C^*(S)\mapsto\dot x\in C^*(S)/K$ restricts to a  
completely isometric map on $S$, contradicting the hypothesis that $S$ is reduced.   
\end{proof}

\begin{cor}\label{pcCor2}
Every irreducible operator system $S\subseteq \mathcal B(H)$ acting on a finite dimensional 
Hilbert space is reduced in the sense of Definition \ref{inDef1}, and 
the identity representation of $\mathcal B(H)$ is a boundary 
representation for $S$.  
\end{cor}

\begin{proof} Since $S$ is an irreducible self adjoint family of operators that 
contains $\mathbf 1$, the 
double commutant theorem implies  
$C^*(S)=\mathcal B(H)$.  Every operator system that generates a simple \cstar\ must 
be reduced, since the only candidate for the boundary ideal is $\{0\}$.  In this 
case, the identity representation of 
$C^*(S)$ is, up to equivalence, the only irreducible representation, 
so that Theorem \ref{pcThm1} implies that it must be a boundary representation.  
\end{proof}

\begin{rem}[Fixed points of UCP maps on matrix algebras]\label{pcRem1}  
Corollary 
\ref{pcCor2} is equivalent to the following assertion: If $\phi: \mathcal B(H)\to \mathcal B(H)$ 
is a UCP map on a matrix algebra that fixes an irreducible set $S$ of operators, 
then $\phi$ is the identity map.  When $S=\{a\}$ consists 
of a single irreducible operator $a$, for example, there appears to 
be no direct route to a proof of the assertion.  
In the special case where $\phi$ preserves the tracial state of $\mathcal B(H)$, 
a straightforward application of the Schwarz inequality 
implies that the operator system $S=\{x\in\mathcal B(H): \phi(x)=x\}$ 
is closed under operator multiplication, and that implies $\phi$ is 
the identity map when $S$ is irreducible.  

But in general, the best one can say is that 
there is a UCP idempotent $E$ with the same fixed elements - 
recall that the set of accumulation points 
of the sequence $\phi^k$, $k=1,2,\dots$ contains a unique idempotent $E$.  
Once one has such an $E$ one can introduce the Choi-Effros multiplication \cite{ChE1} 
on $S$ to make it into a $C^*$-algebra.  However, since the Choi-Effros multiplication need 
not be the ambient multiplication in $\mathcal B(H)$, this fails to address the 
issue.  
Thus, Theorem \ref{pcThm1} and Corollary \ref{pcCor2} make 
significant assertions even for operator systems 
in matrix algebras.  
\end{rem}

Finally, we recall the key property of reduced operator systems in 
general, which follows from Theorem 2.2.5 of \cite{arvSubalgI} together with Theorem \ref{pcThm1}: 

\begin{thm}\label{pcThm2}
Let $S_1\subseteq C^*(S_1)$ and $S_2\subseteq C^*(S_2)$ be two reduced separable operator 
systems.  Then every isomorphism of operator systems $\theta: S_1\to S_2$ extends 
uniquely to a $*$-isomorphism of \cstar s $\tilde\theta: C^*(S_1)\to C^*(S_2)$.  
\end{thm}

\section{Peaking and boundary representations}\label{S:pr}

Let $X$ be a compact Hausdorff space.  A {\em function system} is a linear 
subspace of $C(X)$ that separates points, is closed under complex conjugation, 
and contains the constants.  
A {\em peak point} for a function system $S\subseteq C(X)$ is a point 
$x\in X$ with the property that there is a ``peaking function" $f\in S$ 
(which of course depends on $x$) such that 
$$
|f(x)| > |f(y)|, \qquad \forall\ y\in X, \quad y\neq x.  
$$ 
Since function systems are spanned by their real functions, one can reformulate 
this condition so as to get rid of absolute values; but the form given is the 
one that we choose for the following discussion.   For separable function algebras, 
a theorem of Bishop and de Leeuw \cite{BdL} asserts that 
the peak points are exactly the points of the Choquet boundary, while for 
the more general separable 
function systems, the peak points are dense in the Choquet boundary (see \cite{phelps} or \cite{phelps2}, 
pp. 39--40 and Corollary 8.4).  
 
We now recall the definition of peaking representations 
from \cite{arvChoq2}.  Given an operator system $S\subseteq C^*(S)$, 
every representation 
$\pi$ of $C^*(S)$ on a Hilbert space $K$ gives rise to a representation 
of the matrix hierarchy over $C^*(S)$, in which for an $n\times n$ matrix 
$(x_{ij})$ of elements of $C^*(S)$, the $n\times n$ operator matrix 
$(\pi(x_{ij}))$ represents an operator on the direct sum $n\cdot K$ of $n$ 
copies of $K$, and the map $(x_{ij})\mapsto (\pi(x_{ij}))\in \mathcal B(n\cdot K)$ 
is a representation of $M_n(C^*(S))$ on $n\cdot K$.

\begin{defn}\label{prDef1}
An irreducible representation $\pi: C^*(S)\to \mathcal B(K)$ is said to be 
{\em peaking} for $S$ if there is an $n=1,2,\dots$ and an 
$n\times n$ matrix $(s_{ij})\in M_n(S)$ with entries in $S$ 
that has the following property:  
For every irreducible representation $\sigma$ of $C^*(S)$ that is 
inequivalent to $\pi$, one has 
\begin{equation}\label{prEq1}
\|(\pi(s_{ij}))\|>\|(\sigma(s_{ij}))\|.  
\end{equation}
Such an operator matrix $(s_{ij})\in M_n(S)$ is called a {\em peaking operator} for $\pi$.  
\end{defn}

Actually, in \cite{arvChoq2} it was necessary to 
employ a stronger variant of this concept, called {\em strong peaking}.  
That stronger property 
will not be required 
here since the \cstar s of this paper are finite dimensional, with only 
a finite number of inequivalent irreducible representations.  We will make 
use the following special case of a general result in \cite{arvChoq2}:  

\begin{thm}\label{prThm1}
Let $S$ be an operator system that generates a finite dimensional \cstar \ 
$C^*(S)$.  An irreducible representation of $C^*(S)$ is a boundary representation 
for $S$ iff it is peaking for $S$.  
\end{thm}

\begin{proof}
The decomposition (\ref{pcEq2}) of $C^*(S)$ into a central direct sum of 
matrix algebras implies that there are exactly $N$ mutually inequivalent 
irreducible representations $\pi_1,\dots,\pi_N$ of $C^*(S)$.  So for example, 
$\pi_1$ is peaking for $S$ iff there is a $p\geq 1$ and a $p\times p$ matrix 
$(s_{ij})$ over $S$ such that 
$$
\|(\pi_1(s_{ij}))\|>\max_{2\leq k\leq N}\|(\pi_k(s_{ij}))\|.  
$$
Theorem 6.2 of \cite{arvChoq2} implies that this assertion 
is equivalent 
to the assertion that $\pi_1$ is a boundary representation for $S$.  
\end{proof}

\section{Proof of Theorem \ref{inThm1}}\label{S:p1}

In this section we prove Theorem \ref{inThm1}.  
Choose an $N$-tuple of unit preserving self adjoint linear 
maps $\Gamma_k: Z\to \mathcal B(H_k)$, $k=1,\dots, N$, satisfying 
properties (i)--(iii) of Section \ref{S:in}, let ${\bf \Gamma}=(\Gamma_1,\dots, \Gamma_N)$ 
and consider the operator system  
$
S_{\bf \Gamma}\subseteq \mathcal B(H_1\oplus\cdots\oplus H_N)
$
defined by 
$$
S_{\bf \Gamma}=\{\Gamma_1(z)\oplus\cdots\oplus\Gamma_n(z): z\in Z\}.  
$$
We will show that 
\begin{equation}\label{p1Eq1}
C^*(S_{\bf \Gamma})=\mathcal B(H_1)\oplus\cdots\oplus\mathcal B(H_N), 
\end{equation}
that each of the visible irreducible representations $\pi_k: C^*(S)\to \mathcal B(H_k)$ 
$$
\pi_k(x_1\oplus\cdots\oplus x_k)=x_k, \qquad k=1,\dots, N
$$
is a boundary representation for $S_{\bf \Gamma}$, and we will deduce 
that $S_{\bf \Gamma}$ is a reduced operator system with the asserted integer 
invariants.  

To prove (\ref{p1Eq1}), note first that each of the representations $\pi_1,\dots, \pi_N$ 
of $C^*(S_{\bf \Gamma})$ is irreducible, since $\pi_k(C^*(S_{\bf \Gamma}))$ contains the 
irreducible space of operators $\Gamma_k(Z)\subseteq \mathcal B(H_k)$.  Moreover, the strong separation 
property (iii) implies that the representations $\pi_k$ are mutually inequivalent.  It 
follows that they are mutually disjoint; and since the identity representation of 
$C^*(S_{\bf \Gamma})$ is the direct sum 
$\pi_1\oplus\cdots\pi_N$,  it follows that the projection $p_k$ of $H_1\oplus\cdots\oplus H_N$ 
onto each summand $H_k$ belongs to $C^*(S_{\bf \Gamma})$.  Formula (\ref{p1Eq1}) follows.  

Finally, let $L: Z\to \mathcal B(H_1\oplus\cdots\oplus H_N)$ be the linear 
map 
$$
L(z)=\Gamma_1(z)\oplus\cdots\oplus \Gamma_N(z), \qquad z\in Z.  
$$
We have 
 $S_{\bf\Gamma}=L(Z)$ by definition of $S_{\bf\Gamma}$, and moreover  
$$
\Gamma_k(z)=\pi_k(L(z)), \qquad z\in Z, \quad k=1,\dots, N.  
$$
Using the latter formula, one finds that the strong separation property (iii) implies that 
each representation $\pi_k$ is peaking for $S_{\bf\Gamma}$.  Theorem \ref{prThm1} 
implies that each $\pi_k$ is a boundary representation for $S$, and since 
$\{\pi_1,\dots, \pi_N\}$ is a complete list of the irreducible representations of 
$C^*(S_{\bf\Gamma})$ up to equivalence, 
Corollary \ref{pcCor1} implies that $S_{\bf\Gamma}$ is reduced.

To prove the converse assertion of Theorem \ref{inThm1}, 
let $S\subseteq \mathcal B(H)$ be a reduced operator system acting 
on a finite dimensional Hilbert space $H$.  We have to  show that there is 
an $N$-tuple  ${\bf\Gamma}=(\Gamma_1,\dots, \Gamma_N)$ of linear maps 
with the stated properties such 
that $S$ and $S_{\bf\Gamma}$ are isomorphic operator systems.  

To do that, consider the natural decomposition of $C^*(S)$ 
$$
C^*(S)=A_1\oplus\cdots\oplus A_N, \qquad A_k\cong\mathcal B(H_k), \quad k=1,\dots, N  
$$
into finite dimensional subfactors.  
We first exhibit an appropriate parameterization for $S$ in terms 
of the parameter space $Z=\mathbb C^d$, $d=\dim S$.  For that we   
claim that there is a linear basis $s_1,\dots, s_d$ for $S$ consisting of self adjoint 
operators $s_k$ that satisfies $s_1+\cdots+s_d=\mathbf 1$.  Indeed, if we choose a 
$d-1$ dimensional subspace $S_0$ of $S$ such that $S_0^*=S_0$ and $\mathbf 1\notin S_0$, 
then we obtain such a basis by choosing a linear basis $s_1,\dots, s_{d-1}$ for $S_0$ 
consisting of self adjoint operators and setting $s_d=\mathbf 1-(s_1+\cdots+s_{d-1})$.  
Set $Z=\mathbb C^d$.  Then we have arranged 
that the linear map $L: Z\to S$ defined by 
$$
L(z_1,\dots, z_d)=z_1\cdot s_1+\cdots+ z_d\cdot s_d
$$
is linear isomorphism of vector spaces that preserves the 
$*$ operation and carries $(1,1,\dots, 1)$ to the identity operator of $S$.  

Let $\pi_k: C^*(S)\to \mathcal B(H_k)$ be the representation associated with 
the $k$th term in the decomposition (\ref{p1Eq1}) and define $\Gamma_k: Z\to \mathcal B(H_k)$ 
by 
\begin{equation}\label{p1Eq2}
\Gamma_k(z)=\pi_k(L(z)), \qquad z\in Z.  
\end{equation}
Each $\Gamma_k$ is a linear map that preserves the $*$-operation and maps 
the vector $(1,1,\dots, 1)\in Z$ to the operator $\mathbf 1_{H_k}$.  
We have $\Gamma_k(Z)=\pi_k(S)$, so that $C^*(\Gamma_k(Z))=\pi_k(C^*(S))=\mathcal B(H_k)$, 
and hence  $\Gamma_k(Z)$ is an irreducible operator system in $\mathcal B(H_k)$, 
$k=1,\dots, N$.  Note too 
that these maps $\Gamma_1,\dots,\Gamma_N$ satisfy property (ii) of Section \ref{S:in} since 
if $\Gamma_k(z)=0$ for every $k$, then $\pi_k(L(z))=0$ for every $k=1,\dots, N$, 
hence $L(z)=0$ since $\ker \pi_1\cap\cdots\cap \ker \pi_N=\{0\}$, and finally 
$z=0$ since $L$ is an isomorphism of vector spaces.  

We claim that the $N$-tuple ${\bf \Gamma}=(\Gamma_1,\dots, \Gamma_N)$ satisfies 
property (iii) of Section \ref{S:in}.  Indeed, by Corollary \ref{pcCor1}, every 
representation $\pi_k: C^*(S)\to \mathcal B(H_k)$ is a boundary representation for 
$S$ which,  by Theorem \ref{prThm1}, is a peaking representation for $S$.  Property 
(iii) now follows after one unravels that assertion through each of 
the parameterizations 
$\Gamma_k(z)=\pi_k(L(z))$, $k=1,\dots, N$, in terms of the basic map $L: Z\to S$.  

Thus if we form the operator system associated with the $N$-tuple 
${\bf\Gamma}=(\Gamma_1,\dots, \Gamma_N)$ 
$$
S_{\bf\Gamma}=
\{\Gamma_1(z)\oplus\cdots\oplus\Gamma_N(z): z\in Z\}\subseteq \mathcal B(H_1\oplus\cdots\oplus H_N),    
$$
then formula (\ref{p1Eq1}) implies that 
$$
C^*(S_{\bf\Gamma})=\mathcal B(H_1)\oplus\cdots\oplus \mathcal B(H_N).  
$$

Now for each $k=1,\dots, N$, the irreducible representation $\pi_k: A_k\to \mathcal B(H_k)$ 
is an isomorphism of \cstar s, so we can define a $*$-isomorphism of \cstar s 
$\pi: C^*(S)\to C^*(S_{\bf\Gamma})$ by way of 
$$
\pi(x_1\oplus\cdots\oplus x_N)=\pi_1(x_1)\oplus\cdots\oplus \pi_N(x_N), \quad x_k\in A_k, 
\quad k=1,\dots,N.  
$$
This isomorphism satisfies $\pi(L(z))=\Gamma_1(z)\oplus\cdots\oplus \Gamma_N(z)$ 
for all $z\in Z$, and therefore $\pi(S)=S_{\bf\Gamma}$.  Hence $S$ and $S_{\bf\Gamma}$ 
are isomorphic operator systems.

\section{Proof of Theorem \ref{inThm2}}\label{S:p2}

Let ${\bf\Gamma}=(\Gamma_1,\dots, \Gamma_N)$ and ${\bf \tilde\Gamma}=(\tilde\Gamma_1,\dots,\tilde\Gamma_N)$ 
be two parameterizing sequences of linear maps defined on $Z=\mathbb C^d$ 
of the same length  that satisfy (i), (ii),  (iii) of Section \ref{S:in}, 
and assume that  $S_{\bf \Gamma}$ and $S_{\bf \tilde\Gamma}$ 
are isomorphic operator systems.  We will show that ${\bf \Gamma}$ and 
${\bf\tilde\Gamma}$ are equivalent parameter sequences.  Since $S_{\bf\Gamma}$ and 
$S_{\bf\tilde\Gamma}$ are reduced operator systems, Theorem \ref{pcThm2} implies that there 
is a $*$-isomorphism $\theta: C^*(S_{\bf \Gamma})\to C^*(S_{\bf \tilde \Gamma})$ such that 
$\theta(S_{\bf\Gamma})=S_{\bf \tilde\Gamma}$.  After consideration of the central decompositions 
of these \cstar s 
into full matrix algebras 
$$
C^*(S_{\bf\Gamma})=A_1\oplus\cdots\oplus A_n, 
\qquad C^*(S_{\bf\tilde\Gamma})=\tilde A_1\oplus\cdots\oplus A_N, 
$$
$A_k\cong\mathcal B(H_k)$, $\tilde A_k\cong\mathcal B(\tilde H_k)$, 
it follows that there is a permutation $\sigma$ of $\{1,2,\dots, N\}$ such that 
$\theta(A_k)=\tilde A_{\sigma(k)}$, $k=1,\dots, N$.   For each $k$, the restriction of 
$\theta$ to $A_k\cong\mathcal B(H_k)$ can be viewed as a $*$-isomorphism 
of $\mathcal B(H_k)$ onto $\mathcal B(\tilde H_{\sigma(k)})$, which  
is implemented by a unitary operator $U_k: H_k\to \tilde H_{\sigma(k)}$
$$
\theta(a)=U_kaU_k^{-1}, \qquad a\in A_k, \quad k=1,\dots,N.   
$$
We conclude that 
\begin{align*}
\theta(S_{\bf\Gamma})
&=\{U_1\Gamma_1(z)U_1^{-1}\oplus\cdots\oplus U_N\Gamma_N(z)U_N^{-1}: z\in Z\}
\\
&=S_{\bf\tilde\Gamma}=\{\tilde\Gamma_{\sigma(1)}(z)\oplus\cdots\oplus \tilde\Gamma_{\sigma(N)}(z):z\in Z\}.  
\end{align*}
Since $\ker\Gamma_1\cap\cdots\cap\ker\Gamma_N=\ker\tilde\Gamma_1\cap\cdots\cap\ker\tilde\Gamma_N=\{0\}$, 
the equality of these two spaces of operators 
implies that for every $z\in Z$ there is 
a unique vector $\alpha(z)\in Z$ such that 
\begin{equation}\label{p2Eq1}
\tilde\Gamma_{\sigma(k)}(\alpha(z))=U_k\Gamma_k(z)U_k^{-1}, \qquad z\in Z, \quad k=1,\dots, N.   
\end{equation}
Moreover, since all of the maps $\Gamma_k$ and $\tilde\Gamma_k$ are complex linear, 
preserve the $*$ operation 
and map the unit of $Z$ to the corresponding identity operator, (\ref{p2Eq1}) implies 
that $\alpha$ must in fact 
be a unit preserving automorphism of the $*$-vector space structure of $Z$.  Hence 
(\ref{p2Eq1}) shows that the two parameter sequences ${\bf \Gamma}$ and ${\bf\tilde \Gamma}$ 
are equivalent.  

The proof of the converse is a straightforward reversal of this argument.

\section{Structure of nonreduced operator systems}\label{S:nr} 

We conclude with a discussion of how one constructs nonreduced operator systems 
using these methods.  We will outline the construction - giving precise 
definitions but no proofs - sketching the proof of only a single key lemma to illustrate 
the technique.

Let $S$ be an operator system that generates a finite dimensional \cstar\ $C^*(S)$, 
let $K$ be the boundary ideal for $S$, let $\dot S\subseteq C^*(S)/K$ be the 
corresponding reduced operator system in the $C^*$-envelope of $S$, 
and consider the central decomposition of 
$C^*(S)$ into factors
$$
C^*(S)=A_1\oplus\cdots\oplus A_s, \qquad A_k\cong\mathcal B(H_k).  
$$
A complete list of irreducible representations 
$\pi_k: C^*(S)\to\mathcal B(H_k)$ of $C^*(S)$ is associated with 
the minimal central projections.  Some of these 
representations are boundary representations and others are not.  Let us 
relabel so as to collect 
the boundary representations together with the first $N$ of the summands $A_1,\dots, A_N$ and the others 
as the remaining $n-N=M$ terms $A_{N+1},\dots, A_{N+M}$.  It follows that the boundary ideal is 
$\ker\pi_1\cap\cdots\cap\ker\pi_N$,  
$$
K=0\oplus\cdots\oplus 0\oplus A_{N+1}\oplus\cdots \oplus A_{N+M}, 
$$ 
and the quotient $C^*(S)/K$ is, in this finite dimensional setting, 
isomorphic to the remaining summand 
$$
C^*(S)/K\cong A_1\oplus\cdots\oplus A_N\oplus 0\oplus\cdots\oplus 0,   
$$
the quotient map being identified with the obvioous homomorphism of $C^*(S)$ on 
this initial segment.  We conclude from these observations that 
the most general nonreduced operator system is obtained in the following 
way: Let $A\oplus K$ be a direct sum of two 
finite dimensional \cstar s, and let $S$ be a linear subspace 
of $A\oplus K$ with the following properties:
\begin{enumerate}
\item[(i)]  $A\oplus K$ is the \cstar\ generated by $S$.  
\item[(ii)] Every irreducible representation of $A$ is a boundary representation for $S$.  
\item[(iii)] No irreducible representation of $K$ is a boundary representation for $S$.  
\end{enumerate}
Then as we have argued above, the boundary ideal for $S$ is $K$ and $A$ is identified with 
the $C^*$-envelope of $S$.  

Of course, we have not spelled out explicitly how one constructs all such configurations, but 
it is not hard to do so.  To sketch the details, 
let $Z=\mathbb C^d$ be a unital $*$-vector space as in section 
\ref{S:in}, and let 
$$
{\bf\Gamma}=(\Gamma_1,\dots, \Gamma_N), \qquad {\bf\Omega}=(\Omega_1,\dots,\Omega_M)
$$ 
be two tuples of unit-preserving $*$-preserving linear maps from the parameter space 
$Z$ to 
$\mathcal B(H_1), \dots, \mathcal B(H_N)$ and  
$\mathcal B(K_1),\dots, \mathcal B(K_M)$ 
respectively, such that the range of each of the $M+N$ maps is an irreducible operator system.  
One can use ${\bf\Gamma}=(\Gamma_1,\dots, \Gamma_N)$ to construct the 
$C^*$-envelope of an 
operator system and  
${\bf\Omega}=(\Omega_1,\dots, \Omega_M)$ to construct its boundary ideal in the following way.  
The required properties are:  
\begin{enumerate}
\item[(a)] (Subordination of ${\bf\Omega}$ to ${\bf\Gamma}$) 
For every $r=1,\dots,M$, every $p=1,2,\dots$,  and 
every $p\times p$ matrix $(z_{ij})$ over $Z$, 
$$
\|(\Omega_r(z_{ij}))\|\leq \max_{1\leq k\leq N}\|(\Gamma_k(z_{ij}))\|.   
$$
\item[(b)] (Strong separation in the components of ${\bf\Gamma}$) For every $k=1,\dots, N$, 
there is a $p=p(k)=1,2,\dots$ and a 
$p\times p$ matrix $(z_{ij})=(z_{ij}(k))$ with components in $Z$ such that 
$$
\|(\Gamma_k(z_{ij}))\|>\max_{l\neq k}\|(\Gamma_l(z_{ij}))\|.  
$$
\item[(c)] (Weak separation in all components) For every $1\leq r<s\leq M$, 
there is a $p=p(r,s)=1,2,\dots$ and a 
$p\times p$ matrix $(z_{ij})=(z_{ij}(r,s))$ over $Z$ such that 
$$
\|(\Omega_r(z_{ij}))\|\neq \|(\Omega_s(z_{ij}))\|,   
$$
and for every $1\leq r\leq M$ and every $1\leq k\leq N$, 
there is a $p=p(r,k)=1,2,\dots$ and a 
$p\times p$ matrix $(z_{ij})=(z_{ij}(r,k))$ over $Z$ such that 
$$
\|(\Omega_r(z_{ij}))\|\neq \|(\Gamma_k(z_{ij}))\|,   
$$
\end{enumerate}

We can assemble the component maps of ${\bf\Gamma}$ and ${\bf \Omega}$ 
to define 
an operator system (\ref{nrEq1}).  The following result shows 
that the \cstar\ generated by  
this operator system is the expected direct sum of matrix algebras: 

\begin{lem}\label{nrLem1}
Let ${\bf\Gamma}=(\Gamma_1,\dots, \Gamma_N)$  and 
${\bf\Omega}=(\Omega_1,\dots, \Omega_M)$ 
be two sequences of 
self adjoint unit preserving linear 
maps 
$$
\Gamma_k: Z\to \mathcal B(H_k), \qquad \Omega_r: Z\to \mathcal B(K_r)
$$
such that each $\Gamma_k(Z)$ and each 
$\Omega_r(Z)$ is an irreducible operator system, 
and which together satisfy the weak separation 
property (c).  Let $S_{\bf\Gamma,\Omega}$ be the associated operator system 
in $\mathcal B(H_1\oplus\cdots\oplus H_N\oplus K_1\oplus\cdots\oplus K_M)$ defined by 
\begin{equation}\label{nrEq1}
S_{\bf\Gamma,\Omega}=
\{\Gamma_1(z)\oplus\cdots\oplus \Gamma_N(z)\oplus\Omega_1(z)\oplus\cdots\oplus\Omega_M(z): z\in Z\}.
\end{equation}
Then $C^*(S_{\bf\Gamma,\Omega})=\mathcal B(H_1)\oplus\cdots\oplus \mathcal B(H_M)
\oplus\mathcal B(K_1)\oplus\cdots\oplus \mathcal B(K_M)$.  
\end{lem}

\begin{proof}[Sketch of proof]
For every $r=1,\dots, M+N$, let $\pi_r$ 
be the representation of $C^*(S_{\bf\Gamma,\Omega})$ defined by 
\begin{equation}\label{nrEq2}
\pi_r(x_1\oplus\cdots\oplus x_{M+N})=x_r, \qquad 1\leq r\leq M+N.  
\end{equation}
$\pi_r$ is an irreducible representation because its range contains
the irreducible operator system $\Gamma_r(Z)$ if $1\leq r\leq N$ or 
$\Omega_{r-N}(Z)$ if $N<r\leq M+N$.  The hypothesis (c) implies 
that $\pi_1, \dots, \pi_{M+N}$ are mutually inequivalent, so 
by irreducibility, they are mutually disjoint.  It follows that 
for every $r$, the projection 
corresponding to the $r$th summand of the \cstar 
$$
\mathcal B(H_1)\oplus\cdots\oplus\mathcal B(H_N)\oplus\mathcal B(K_1)\oplus\cdots\oplus \mathcal B(K_M)
$$ 
belongs to the center of  $C^*(S_{\bf\Gamma,\Omega})$.  The assertion is now immediate.   
\end{proof}

Now assume that properties (a), (b) and (c) are satisfied, 
let $S_{\bf\Gamma,\Omega}$ be the operator system (\ref{nrEq1}) and 
let $A$ and $K$ be the summands of $C^*(S_{\bf\Gamma,\Omega})$ defined by 
\begin{align*}
A&=\mathcal B(H_1)\oplus\cdots\oplus \mathcal B(H_N)\oplus 0\oplus\cdots\oplus 0\\
K&=0\oplus\cdots\oplus 0\oplus\mathcal B(K_1)\oplus\cdots\oplus\mathcal B(K_M),   
\end{align*}
and let $\pi_1,\dots, \pi_{M+N}$ be the list of irreducible representations (\ref{nrEq2}).  
Then minor variations of arguments already given show that 
$\pi_1,\dots, \pi_N$ are all peaking representations for $S_{\bf\Gamma, \Omega}$, whereas 
none of $\{\pi_{N+1},\dots, \pi_{N+M}\}$ is peaking for $S_{\bf\Gamma,\Omega}$.  It follows 
from Corollary \ref{pcCor1} and Theorem \ref{prThm1} 
that $K$ is the boundary 
ideal for $S_{\bf\Gamma,\Omega}$ and that $A$ is identified with the 
$C^*$-envelope of $S_{\bf\Gamma,\Omega}$ in the manner described above.  
Moreover, every (nonreduced) 
operator system which generates a finite dimensional \cstar\ is isomorphic to 
one obtained from the above construction.  We omit those details.

\bibliographystyle{alpha}

\newcommand{\noopsort}[1]{} \newcommand{\printfirst}[2]{#1}
  \newcommand{\singleletter}[1]{#1} \newcommand{\switchargs}[2]{#2#1}

\end{document}